\documentclass{amsart}
\usepackage{amssymb}
\usepackage{latexsym}
\usepackage{amsthm}
\usepackage{amscd}
\theoremstyle{plain}

\newtheorem{thm}{Theorem}[section]

\newtheorem{lem}[thm]{Lemma}

\newtheorem{prop}[thm]{Proposition}

\newtheorem{defn}{Definition}[section]
\newtheorem{rem}{Remark}[section]

\numberwithin{equation}{section}

\newcommand{\EE}{\mathcal{E}}

\newcommand{\ZZ}{\mathbb Z}
\newcommand{\CC}{\mathbb C}

\newcommand{\RR}{\mathbb R}

\newcommand{\be}{{\bf e}}
\newcommand{\ba}{\vec{a}}
\newcommand{\bb}{\vec{b}}
\newcommand{\bc}{\vec{c}}
\newcommand{\bx}{\vec{x}}
\newcommand{\by}{\vec{y}}
\newcommand{\bz}{\vec{z}}
\newcommand{\PP}{\wp}
\newcommand{\HH}{\mathbb H}

\newcommand{\csp}{\null\hskip 20pt}
\newcommand{\ccsp}{\null\hskip 40pt}
\newcommand{\cccsp}{\null\hskip 80pt}
\newcommand{\ccccsp}{\null\hskip 160pt}

\allowdisplaybreaks[1]

\begin{document}

\title[ELLIPTIC APOSTOL-DEDEKIND SUMS GENERATE DEDEKIND SYMBOLS]
{The elliptic Apostol-Dedekind sums generate odd Dedekind symbols with
Laurent polynomial reciprocity laws}
\author{Shinji Fukuhara}
\thanks{\it{Address.} \rm{Department of Mathematics,
  Tsuda College, Tsuda-machi 2-1-1, \\
  Kodaira-shi, Tokyo 187-8577, Japan
  (e-mail: fukuhara@tsuda.ac.jp).}
}

\begin{abstract}
Dedekind symbols are generalizations of the classical Dedekind sums (symbols).
There is a natural isomorphism between the space of Dedekind symbols
with Laurent polynomial reciprocity laws and the space of modular forms.
We will define a new elliptic analogue of the Apostol-Dedekind sums. Then
we will show that
the newly defined sums generate all odd Dedekind symbols with Laurent polynomial
reciprocity laws.
Our construction is based on Machide's result \cite{MA1} on his elliptic
Dedekind-Rademacher sums.
As an application of our results, we discover Eisenstein series identities
which generalize certain formulas by Ramanujan\cite{RA1},
van der Pol \cite{PO1}, Rankin\cite{RA2}
and Skoruppa \cite{SK1}.
\end{abstract}

\maketitle

\section{Introduction and statement of results}\label{sect1}

A {\em Dedekind symbol} is a generalization
of the classical Dedekind sums (\cite{RG1}),
and is defined
as a complex valued function $D$ on
  $V:=\{(p,q)\in \ZZ^+\times \ZZ\,|\,\gcd(p,q)=1\}$
satisfying
\begin{equation}\label{eqn1.1}
  D(p,q)=D(p,q+p).
\end{equation}
The symbol $D$ is determined uniquely by its {\em reciprocity law}:
\begin{equation}\label{eqn1.2}
  D(p,q)-D(q,-p)=R(p,q)
\end{equation}
up to an additive constant.
The function $R$ is defined on
  $U:=\{(p,q)\in \ZZ^+\times \ZZ^+\,|\,\gcd(p,q)=1\},$
and is called a {\em reciprocity function} associated with the
Dedekind symbol $D$.
The function $R$ necessarily satisfies
the equation:
\begin{equation}\label{eqn1.3}
  R(p+q,q)+R(p,p+q)=R(p,q).
\end{equation}
When the reciprocity function $R$ is a Laurent polynomial in $p$ and $q$,
the symbol $D$ is called a {\em Dedekind symbol with Laurent polynomial
reciprocity law}. These symbols are particularly important because
they naturally correspond to modular forms (\cite{FU1}).
The symbol $D$ is said to be {\em even} (resp. {\em odd}) if $D$
satisfies:
\begin{equation}\label{eqn1.4}
  D(p,-q)=D(p,q)\ \ \ (\text{resp.}\ D(p,-q)=-D(p,q)).
\end{equation}

To state our results, we need to review the relevant relationship between
modular forms, Dedekind symbols and period polynomials
(see \cite{FU1} for details).
Throughout the paper,
we assume that $w$ is an {\em even} positive integer,
and we use the following notation:
\begin{align*}
  M_{w+2}&:=
    \text{the vector space of modular forms on $SL_2(\ZZ)$ with weight $w+2$}, \\
  S_{w+2}&:=
    \text{the vector space of cusp forms on $SL_2(\ZZ)$ with weight $w+2$}, \\
  d_w&:=\begin{cases}
       \left[\frac{w+2}{12}\right]-1
         & \mathrm{if\ \ \ } w\equiv 0 \pmod {12} \\
       \left[\frac{w+2}{12}\right]
         & \mathrm{if\ \ \ } w\not\equiv 0 \pmod {12}
       \end{cases}
\end{align*}
where $[x]$ denotes the greatest integer not exceeding $x\in\RR$.
We note that
$$\dim S_{w+2}=d_w \text{\ \ \ and\ \ \ } \dim  M_{w+2}=d_w+1.$$
Let $B_k$ denote the $k$th Bernoulli number,
and let $g_w$ be a Laurent polynomial in $p$ and $q$ defined by
  $$g_w(p,q):=
    -\frac{1}{pq}
    \left\{
    \sum^{\frac{w}{2}+1}_{j=0}
    \frac{w!B_{2j}B_{w+2-2j}}{2(2j)!(w+2-2j)!}p^{2j}q^{w+2-2j}
    +\frac{B_{w+2}}{2(w+2)}
    \right\}.$$
We also use the following notation:
{\allowdisplaybreaks
  \begin{align*}
  \mathcal{V}_w^-&:=\{g|\
  \text{$g$ is an odd homogeneous polynomial in $p$ and $q$ of degree $w$ which}
  \\
  &\ccsp
    \text{\ \ \ satisfies $g(p+q,q)+g(p,p+q)=g(p,q)$ and $g(p,q)=g(q,p)$}\} \\
  &\ccsp\csp
    \text{\ \ (an element of $\mathcal{V}_w^-$ is essentially
    an odd period polynomial)},
  \\
  \mathcal{W}_w^-&:=\mathcal{V}_w^-\oplus \CC(g_w)
  \text{\ \ \ ($\CC(g_w)$ is the vector space spanned by $g_w$)},
  \\
  \mathcal{D}_w^-&:=\{D\,|\
  \text{$D$ is an odd Dedekind symbol
    such that $D(p,q)-D(q,-p)\in\mathcal{W}_w^-$}\}.
  \end{align*}
}

First we will see that the three spaces $M_{w+2}$, $\mathcal{D}_w^-$
and $\mathcal{W}_w^-$ are naturally isomorphic.
For a cusp form $f\in S_{w+2}$ and $(p,q)\in V$,
we define $D_f$ and $D_f^-$ by
\begin{equation*}
  D_f(p,q):=\int_{q/p}^{i\infty}f(z)(pz-q)^{w}dz, \ \ \
  D_f^-(p,q):=\frac{1}{2}\{D_f(p,q)-D_f(p,-q)\}.
\end{equation*}
Then we can show $D_f^-$ is an odd  Dedekind symbol in $\mathcal{D}_w^-$
($D_f^-$ can be similarly defined for $f\in M_{w+2}$, see \cite{FU1}
for further details).
Hence we can define a map
\begin{equation*}
  \alpha_{w+2}^-:M_{w+2}\to\mathcal{D}_w^-
\end{equation*}
by
\begin{equation*}
  \alpha_{w+2}^-(f)=D_f^-.
\end{equation*}
Next we define a map
\begin{equation*}
  \beta_{w}^-:\mathcal{D}_w^-\to\mathcal{W}_w^-
\end{equation*}
by
  \begin{equation*}
    \beta_{w}^-(D)(p,q)=D(p,q)-D(q,-p).
  \end{equation*}
In other words, $\beta_{w}^-(D)$ is the reciprocity function of
the Dedekind symbol $D$.

It was shown in \cite{FU1} that these two maps $\alpha_{w+2}^-$ and
$\beta_{w}^-$ are isomorphisms and
$\beta_{w}^-\alpha_{w+2}^-$ can be identified with the Eichler-Shimura
isomorphism. Indeed
$\beta_{w}^-\alpha_{w+2}^-(f)(p,q)$
gives the (homogeneous) odd period polynomial of $f$.

These facts may be summarized in the following commutative diagram:

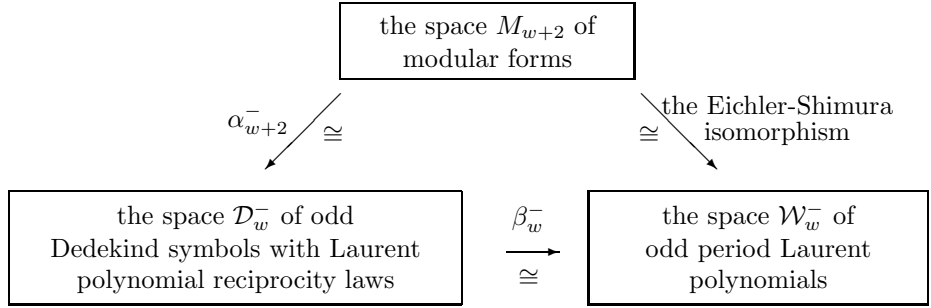
\setlength{\unitlength}{1mm}
\begin{picture}(120,50)(4,-1)
  \put(44,35){\framebox(39,10)
    {\shortstack{the space $M_{w+2}$ of \\
      modular forms}}
  }
  \put(0,5){\framebox(60,15)
    {\shortstack{the space $\mathcal{D}_w^-$ of odd \\
                 Dedekind symbols with Laurent \\
                 polynomial reciprocity laws}}
  }
  \put(77,5){\framebox(45,15)
    {\shortstack{the space $\mathcal{W}_w^-$ of \\
                 odd period Laurent \\
                 polynomials}}
  }
  \put(44,33){\vector(-1,-1){10}}
  \put(28,27){\makebox(10,5){$\alpha_{w+2}^-$}}
  \put(38,25){\makebox(10,5){$\cong$}}
  \put(84,33){\vector(1,-1){10}}
  \put(97,22){\makebox(10,15)
    {\shortstack{the Eichler-Shimura \\
                 isomorphism}}
  }
  \put(80,25){\makebox(10,5){$\cong$}}
  \put(66,12){\vector(1,0){7}}
  \put(64,14){\makebox(10,5){$\beta_{w}^-$}}
  \put(64,6){\makebox(10,5){$\cong$}}
\end{picture}

Next we also need to recall the generalized Dedekind sum defined by
Apostol \cite{AP1}.
The first Dedekind symbol, after the classical Dedekind sum, was
given by Apostol, which we call the Apostol-Dedekind sum
to distinguish it from other generalized Dedekind sums.
Let $k$ be an positive integer, and let $(p,q)$ be in $V$.
The Apostol-Dedekind sum $s_k(q,p)$ is defined by
\begin{equation*}
  s_k(q,p):=\sum_{\mu=1}^{p-1}\frac{\mu}{p}\bar{B}_{k}(\frac{\mu q}{p}).
\end{equation*}
Here $\bar{B}_k(x)$ denotes the $k$th Bernoulli
function.
That is, $\bar{B}_k(x)$ is given by the Fourier expansion
  $$\bar{B}_k(x):=-k!
      \sum^{+\infty}_{\substack{m=-\infty\\ m\ne 0}}
      \frac{e^{2\pi imx}}{(2\pi im)^{k}}.$$
It is well-known that for $0\leq x<1$, $\bar{B}_k(x)$ reduces
to the $k$th Bernoulli polynomial $B_k(x)$.

If $k$ is even, it is easy to see that $s_k(q,p)=0$.
If $k$ is odd, a reciprocity law for the Apostol-Dedekind sums
was obtained by Apostol \cite[p.149]{AP1}:
\begin{equation}\label{eqn1.5}
  p^{w}s_{w+1}(q,p)+q^{w}s_{w+1}(p,q)
  =-2(w+1)g_w(p,q).
\end{equation}

In \cite{FY1} we have proposed an elliptic analogue of
Apostol-Dedekind sums, say $\tilde{s}_{w+1}(q,p;\tau)$. Here
$\tau\in \HH:=\{z\in\CC\ |\ \Im z>0\}$.
These sums satisfy
$$  \lim_{\tau\to i\infty}\tilde{s}_{w+1}(q,p;\tau)=
  s_{w+1}(q,p).$$
However, they have two defects:

(1) they are not real Dedekind symbols, instead they satisfy
$$\tilde{s}_{w+1}(q+2p,p;\tau)=\tilde{s}_{w+1}(q,p;\tau),$$

(2) they are defined in two different ways depending on the parity condition
of $p$ and $q$.

To rectify these defects, we introduce a new kind of
the elliptic Apostol-Dedekind sum.

In what follows, $\sigma(z;\tau)$, $\PP(z;\tau)$ and
$\zeta(z;\tau)$ denote
the Weierstrass sigma, pe and zeta functions,
and $\zeta^{(k)}(z;\tau)$  denotes the $k$th derivative
${\partial^k \zeta(z;\tau)}/{\partial z^k}$
of $\zeta(z;\tau)$.
Furthermore $E_k(\tau)$ denotes the $k$th Eisenstein series
(details of these functions will be given in the section \ref{sect4}).
\begin{defn}\label{defn1.1}
For $(p,q)\in V$, $\tau\in\HH$ and a positive integer $n$,
we define
\begin{align*}
  D_{2n}^-(p,q;\tau):=&\frac{1}{(2\pi i)^2p(2n)!}
    \sum_{\substack{\lambda,\mu=0\\ (\lambda,\mu)\ne (0,0)}}^{p-1}
    \zeta^{(2n)}\left(
    \frac{\lambda+\mu\tau}{p};\tau\right) \\
  &\times\left[
    \zeta\left(
    \frac{q(\lambda+\mu\tau)}{p};\tau\right)
  -E_2(\tau)\frac{q(\lambda+\mu\tau)}{p}
    +2\pi i\frac{q\mu}{p}\right].
\end{align*}
We call $D_{2n}^-(p,q;\tau)$ the elliptic Apostol-Dedekind sum.

For $(p,q)\in U$, $\tau\in\HH$ and a positive integer $n$,
we also define
\begin{align*}
  R_{2n}^-(p,q;&\tau) \\
  :=&-\frac{1}{(2\pi i)^2pq}\Bigg[
        \sum_{j=1}^{n}E_{2j}(\tau)E_{2n+2-2j}(\tau)
        p^{2j}q^{2n+2-2j} \\
        &\ccsp\csp-E_{2n+2}(\tau)(p^{2n+2}+q^{2n+2})
        -(2n+1)E_{2n+2}(\tau)
     \Bigg] \\
  &-\frac{1}{4\pi in}\frac{\partial E_{2n}(\tau)}{\partial \tau}
    (p^{2n-1}q+pq^{2n-1}).
\end{align*}
\end{defn}

Then this sum $D_{2n}^-(p,q;\tau)$ is an odd Dedekind symbol and expressed
without regard to the parities of $p$ and $q$. Furthermore, this sum
is equipped with Laurent polynomial reciprocity law.
We will formulate these findings more precisely as a theorem.
\begin{thm}\label{thm1.1}
\begin{enumerate}
\item
For $(p,q)\in V$, $\tau\in\HH$ and a positive integer $n$,
it holds that
  $$D_{2n}^-(p,q;\tau)=D_{2n}^-(p,q+p;\tau),\ \
   D_{2n}^-(p,-q;\tau)=-D_{2n}^-(p,q;\tau).
  $$
\item
For $(p,q)\in U$, $\tau\in\HH$ and a positive integer $n$,
$D_{2n}^-(p,q;\tau)$ satisfies the following reciprocity law:
  $$D_{2n}^-(p,q;\tau)+D_{2n}^-(q,p;\tau)
    =R_{2n}^-(p,q;\tau).
  $$
\end{enumerate}
\end{thm}

The sum has the following property:
\begin{equation}\label{eqn1.6}
  \lim_{\tau\to i\infty}D_{2n}^-(p,q;\tau)=
  -\frac{(2\pi i)^{2n}}{(2n+1)!}p^{2n}s_{2n+1}(q,p).
\end{equation}
This means that $D_{2n}^-(p,q;\tau)$ is an elliptic analogue
of Apostol-Dedekind sums.

The most striking feature of the newly defined sum is
that the sum ``generates''
all odd Dedekind symbols with Laurent polynomial reciprocity laws.
\begin{thm}\label{thm1.2}
There are $\tau_0,\tau_1,\ldots,\tau_{d_w}\in\HH$ such that
$D_{w}^-(p,q;\tau_i)\ (i=0,1,\ldots,d_w)$ form a basis of
the space $\mathcal{D}_w^-$ of
odd Dedekind symbols with Laurent polynomial reciprocity laws.
\end{thm}

To prove Theorem \ref{thm1.1}, it is convenient to introduce
the generating functions of $D_{2n}^-(p,q;\tau)$ and $R_{2n}^-(p,q;\tau)$.
\begin{defn}\label{defn1.2}
For $(p,q)\in V$, $\tau\in\HH$ and $x\in\RR$, we define
\begin{equation*}
\begin{split}
  D^-(p,q;&\tau;x) \\
  :=&\frac{1}{(2\pi i)^2p}
    \sum_{\substack{\lambda,\mu=0\\ (\lambda,\mu)\ne (0,0)}}^{p-1}
  \left[
    \zeta\left(
    \frac{\lambda+\mu\tau}{p}-x;\tau\right)
  -E_2(\tau)\left(\frac{\lambda+\mu\tau}{p}-x\right)
    +2\pi i\frac{\mu}{p}\right] \\
   &\cccsp\times\left[
    \zeta\left(
    \frac{q(\lambda+\mu\tau)}{p};\tau\right)
  -E_2(\tau)\frac{q(\lambda+\mu\tau)}{p}
    +2\pi i\frac{q\mu}{p}\right]. \\
\end{split}
\end{equation*}

For $(p,q)\in U$, $\tau\in\HH$ and $x\in\RR$, we define
\begin{equation*}
\begin{split}
  R^-(p,q;\tau;x):=&-\frac{1}{(2\pi i)^2}
     \left[
      \zeta(px;\tau)-E_2(\tau)(px)
     \right]
     \left[
      \zeta(qx;\tau)-E_2(\tau)(qx)
     \right] \\
  &+\frac{q}{4\pi ip}
    \left[
    2\frac{\partial\log\sigma(px;\tau)}{\partial \tau}
    -\frac{\partial E_2(\tau)}{\partial \tau}(px)^2
    -\frac{1}{\pi i}E_2(\tau)
    \right] \\
  &+\frac{p}{4\pi iq}
    \left[
    2\frac{\partial\log\sigma(qx;\tau)}{\partial \tau}
    -\frac{\partial E_2(\tau)}{\partial \tau}(qx)^2
    -\frac{1}{\pi i}E_2(\tau)
    \right] \\
  &+\frac{1}{(2\pi i)^2pq}
    \left[
    \PP(x;\tau)+E_2(\tau)
    \right].\\
\end{split}
\end{equation*}
\end{defn}

Then we know that $D^-(p,q;\tau;x)$ and $R^-(p,q;\tau;x)$
are generating functions of
$D_{2n}^-(p,q;\tau)$ and $R_{2n}^-(p,q;\tau)$, respectively.
Our strategy of establishing Theorem \ref{thm1.1} is first to prove
Theorem \ref{thm1.3}, and then derive the assertion of Theorem \ref{thm1.1}
as its corollary.
\begin{thm}\label{thm1.3}
\begin{enumerate}
\item
For $(p,q)\in V$, $\tau\in\HH$ and $x\in\RR$, it holds that
  $$D^-(p,q;\tau;x)=D^-(p,q+p;\tau;x),\ \
   D^-(p,-q;\tau;x)=-D^-(p,q;\tau;x).
  $$
\item
For $(p,q)\in U$, $\tau\in\HH$ and
a sufficiently small real number $x\ne 0$,
$D^-(p,q;\tau;x)$ satisfies the following reciprocity law:
  $$D^-(p,q;\tau;x)+D^-(q,p;\tau;x)=R^-(p,q;\tau;x)+C(\tau)
  $$
\end{enumerate}
where $C(\tau)$ is a constant with respect to $x$.
\end{thm}

\begin{rem}
{\rm
Using the function  $E_k(z;\tau)$ defined by
  $$E_k(z;\tau):=\sum_{\gamma\in \ZZ\tau+\ZZ}
        (\gamma+z)^{-k}|\gamma+z|^{-s}\Big |_{s=0}
  $$
(see Sczech \cite{SC1} for details of this function),
we can express the sums
$D_{2n}^-(p,q;\tau)$ and
$D^-(p,q;\tau;x)$
in Definitions \ref{defn1.1} and \ref{defn1.2} as follows:
\begin{equation*}
  D_{2n}^-(p,q;\tau)=\frac{1}{(2\pi i)^2p}
    \sum_{\substack{\lambda,\mu=0\\ (\lambda,\mu)\ne (0,0)}}^{p-1}
    E_{2n+1}\left(
    \frac{\lambda+\mu\tau}{p};\tau\right)
    E_1\left(
    \frac{q(\lambda+\mu\tau)}{p};\tau\right),
\end{equation*}

\begin{equation*}
  D^-(p,q;\tau;x)
  =\frac{1}{(2\pi i)^2p}
    \sum_{\substack{\lambda,\mu=0\\ (\lambda,\mu)\ne (0,0)}}^{p-1}
    E_1\left(
    \frac{\lambda+\mu\tau}{p}-x;\tau\right)
    E_1\left(
    \frac{q(\lambda+\mu\tau)}{p};\tau\right).
\end{equation*}
}
\end{rem}

As an application of our results,
in the last section, we discover
Eisenstein series identities (Theorem \ref{thm7.1}).
In doing so, we rediscover the formulas by Ramanujan \cite{RA1},
van der Pol \cite{PO1}, Rankin \cite{RA2} and Skoruppa \cite{SK1}.

\section{Machide's reciprocity laws}
\label{sect2}

In this section we recall Machide's result \cite{MA1}
on his elliptic Dedekind-Rademacher sums. His result will play an important role
in proving Theorem \ref{thm1.3}.
We will use some standard notation:
$\be(x):=\exp(2\pi ix)$,
$q:=\be(\tau)$,
  $$\theta(x;\tau):
  =\sum_{n=-\infty}^{\infty}
    \be(\frac12(n+\frac12)^2\tau+(n+\frac12)(x+\frac12)). \\
  $$
We consider the following functions
(refer to \cite{LE1}, \cite{MA1}, \cite{WE1}, \cite{ZA1})
\begin{align*}
F(\xi,\eta;\tau):
  &=\frac{\theta'(0;\tau)\theta(\xi+\eta;\tau)}
    {\theta(\xi;\tau)\theta(\eta;\tau)}, \\
\underline{F}(x,y;\xi;\tau):
  &=\be(y\xi)F(-x+y\tau,\xi;\tau).
\end{align*}
Set
\begin{equation}\label{eqn2.0}
  \underline{F}(x,y;\xi;\tau)
  =\sum_{m=0}^{\infty}\frac{B_m(x,y;\tau)}{m!}(2\pi i)^m\xi^{m-1}.
\end{equation}
The function $B_m(x,y;\tau)$ is called Kronecker's double series
or the elliptic Bernoulli function.
The following expansion of $B_m(x,y;\tau)$
will be used in the later section:
\begin{equation}\label{eqn2.1}
\begin{split}
B_m(x,y;\tau)
  &=m\Bigg[
  \sum_{j=1}^{\infty}(y-j)^{m-1}\frac{\be(-y\tau)q^j}{\be(-x)-\be(-y\tau)q^j}
    \\
  &\ \ \ \ \
  -\sum_{j=1}^{\infty}(y+j)^{m-1}\frac{\be(y\tau)q^j}{\be(x)-\be(y\tau)q^j}
  +y^{m-1}\frac{\be(-x+y\tau)}{\be(-x+y\tau)-1}
  \Bigg]+B_m(y).
\end{split}
\end{equation}

Let $a,a',b,b',c,c'$ be positive integers, and $x,x',y,y',z,z'$ real numbers.
Suppose that
\begin{equation*}
a'z'-c'x'\not\in \gcd(a',c')\ZZ  \text{\ \ \ and\ \ \ }
b'z'-c'y'\not\in \gcd(b',c')\ZZ.
\end{equation*}
Set $(\ba,\bb,\bc):=((a',a),(b',b),(c',c))$ and
$(\bx,\by,\bz):=((x',x),(y',y),(z',z))$.
Machide defined the elliptic Dedekind-Rademacher sum as:
\begin{equation}\label{eqn2.2}
\begin{split}
S_{m,n}^\tau\begin{pmatrix}\ba&\bb&\bc\\ \bx&\by&\bz\end{pmatrix}
  :=\frac{1}{c'}\sum_{\substack{j\pmod c \\ j'\pmod {c'}}}
  &B_m\left(a'\frac{j'+z'}{c'}-x',a\frac{j+z}{c}-x;\frac{a'}{a}\tau\right)\\
  &\times
    B_n\left(b'\frac{j'+z'}{c'}-y',b\frac{j+z}{c}-y;\frac{b'}{b}\tau\right).
\end{split}
\end{equation}
Furthermore he introduced a generating function for $S_{m,n}^\tau$ by
  $$\mathfrak{S}^\tau
  \begin{pmatrix}\ba&\bb&\bc\\ \bx&\by&\bz\\ X&Y&Z\end{pmatrix}
  :=\sum_{m,n=0}^{\infty}\frac{1}{m!n!}
    S_{m,n}^\tau\begin{pmatrix}\ba&\bb&\bc\\ \bx&\by&\bz\end{pmatrix}
    \left(\frac{X}{a}\right)^{m-1}\left(\frac{Y}{b}\right)^{n-1}
  $$
where $Z$ is defined by $Z=-X-Y$.

Under this notation Machide obtained the following reciprocity law
for $\mathfrak{S}^\tau$.

\begin{thm}[Machide\cite{MA1}]\label{thm2.1}
Let $X,Y,Z$ be variables with $X+Y+Z=0$, and $a,a',b,b',c,c'$
positive integers, and $x,y,z$ real numbers.
Let $x',y'$ and $z'$ be real numbers such that
\begin{equation}\label{eqn2.3}
a'y'-b'x'\not\in \gcd(a',b')\ZZ,\ \
a'z'-c'x'\not\in \gcd(a',c')\ZZ,\ \
b'z'-c'y'\not\in \gcd(b',c')\ZZ,\ \
\end{equation}
and let
$(\ba,\bb,\bc)=((a',a),(b',b),(c',c))$,
$(\bx,\by,\bz)=((x',x),(y',y),(z',z))$.

Suppose that the integers $a,b$ and $c$ $($resp. $a',b'$ and $c'$$)$
have no common factor.

Then we have
\begin{equation*}
\mathfrak{S}^\tau\begin{pmatrix}\ba&\bb&\bc\\ \bx&\by&\bz\\ X&Y&Z\end{pmatrix}
+\mathfrak{S}^\tau\begin{pmatrix}\bb&\bc&\ba\\ \by&\bz&\bx\\ Y&Z&X\end{pmatrix}
+\mathfrak{S}^\tau\begin{pmatrix}\bc&\ba&\bb\\ \bz&\bx&\by\\ Z&X&Y\end{pmatrix}
=0.
\end{equation*}
\end{thm}

\section{Reciprocity laws derived from formulas of Machide and Sczech}\label{sect3}

In this section we prove the following proposition,
from which we will deduce Theorem \ref{thm1.3}.
\begin{prop}\label{prop3.1}
For $(p,q)\in U$, $\tau\in\HH$ and a sufficiently small real number $s\ne 0$,
it holds that
\begin{equation}\label{eqn3.1}
\begin{split}
  \frac1p&\sum_{\substack{\lambda,\mu=0\\ (\lambda,\mu)\ne (0,0)}}^{p-1}
    B_1\left(\frac{\lambda}{p}-s,\frac{\mu}{p};\tau\right)
    B_1\left(\frac{q\lambda}{p},\frac{q\mu}{p};\tau\right) \\
  &\cccsp  +\frac1q\sum_{\substack{\lambda,\mu=0\\ (\lambda,\mu)\ne (0,0)}}^{q-1}
      B_1\left(\frac{\lambda}{q}-s,\frac{\mu}{q};\tau\right)
      B_1\left(\frac{p\lambda}{q},\frac{p\mu}{q};\tau\right)  \\
  &=-B_1(ps,0;\tau)B_1(qs,0;\tau)
    +\frac{q}{2p}B_2(ps,0;\tau)
    +\frac{p}{2q}B_2(qs,0;\tau) \\
  &\cccsp  +\frac{1}{2\pi ipq}\frac{\partial B_1(s,0;\tau)}{\partial s}+C(\tau)
\end{split}
\end{equation}
where $C(\tau)$ is a constant with respect to $s$.
\end{prop}

We will give two proofs for Proposition \ref{prop3.1}.
The first proof is our original one which
is derived from Machide's formula (Theorem \ref{thm2.1}).
The second proof is the one proposed by the referee,
and it is brief and elegant
and is based on Sczech's reciprocity law for
elliptic Dedekind sums (\cite{SC1}).
We believe that our original proof is still interesting
in its own right, and it would be applicable to other problems
related to generalized Dedekind sums.

The first
proof of Proposition \ref{prop3.1} rests on the following
lemma.
\begin{lem}\label{lem3.2}
Under the notation and assumptions of Theorem \ref{thm2.1},
we have
\begin{equation}\label{eqn3.2}
-\frac{c}{2b}S_{2,0}^\tau\begin{pmatrix}\bb&\bc&\ba\\ \by&\bz&\bx\end{pmatrix}
+\frac{c}{2a}S_{0,2}^\tau\begin{pmatrix}\bc&\ba&\bb\\ \bz&\bx&\by\end{pmatrix}
=0,
\end{equation}
\begin{equation}\label{eqn3.3}
\frac{b}{2a}S_{2,0}^\tau\begin{pmatrix}\ba&\bb&\bc\\ \bx&\by&\bz\end{pmatrix}
-\frac{b}{2c}S_{0,2}^\tau\begin{pmatrix}\bb&\bc&\ba\\ \by&\bz&\bx\end{pmatrix}
=0,
\end{equation}
\begin{equation}\label{eqn3.4}
\begin{split}
\frac{a}{2b}S_{0,2}^\tau&\begin{pmatrix}\ba&\bb&\bc\\ \bx&\by&\bz\end{pmatrix}
-S_{1,1}^\tau\begin{pmatrix}\ba&\bb&\bc\\ \bx&\by&\bz\end{pmatrix}
+\frac{b}{2a}S_{2,0}^\tau\begin{pmatrix}\ba&\bb&\bc\\ \bx&\by&\bz\end{pmatrix}
\\
&-S_{1,1}^\tau\begin{pmatrix}\bb&\bc&\ba\\ \by&\bz&\bx\end{pmatrix}
-\frac{c}{2b}S_{2,0}^\tau\begin{pmatrix}\bb&\bc&\ba\\ \by&\bz&\bx\end{pmatrix}
\\
&+\frac{c}{a}S_{0,2}^\tau\begin{pmatrix}\bc&\ba&\bb\\ \bz&\bx&\by\end{pmatrix}
-S_{1,1}^\tau\begin{pmatrix}\bc&\ba&\bb\\ \bz&\bx&\by\end{pmatrix}
=0.
\end{split}
\end{equation}
\end{lem}
\begin{proof}
>From Theorem \ref{thm2.1} we have
\begin{equation*}
\begin{split}
  &\sum_{m,n=0}^{\infty}\frac{1}{m!n!}
    S_{m,n}^\tau\begin{pmatrix}\ba&\bb&\bc\\ \bx&\by&\bz\end{pmatrix}
    \left(\frac{X}{a}\right)^{m-1}\left(\frac{Y}{b}\right)^{n-1}X \\
  &\ \ \ +\sum_{m,n=0}^{\infty}\frac{1}{m!n!}
    S_{m,n}^\tau\begin{pmatrix}\bb&\bc&\ba\\ \by&\bz&\bx\end{pmatrix}
    \left(\frac{Y}{b}\right)^{m-1}\left(\frac{Z}{c}\right)^{n-1}X \\
  &\ \ \ +\sum_{m,n=0}^{\infty}\frac{1}{m!n!}
    S_{m,n}^\tau\begin{pmatrix}\bc&\ba&\bb\\ \bz&\bx&\by\end{pmatrix}
    \left(\frac{Z}{c}\right)^{m-1}\left(\frac{X}{a}\right)^{n-1}X \\
  &\ \ \ =0.
\end{split}
\end{equation*}
>From this and the equation $X=-Y-Z$, we know
\begin{equation}\label{eqn3.5}
\begin{split}
  &\sum_{m,n=0}^{\infty}\frac{1}{m!n!}
    S_{m,n}^\tau\begin{pmatrix}\ba&\bb&\bc\\ \bx&\by&\bz\end{pmatrix}
    (-1)^ma^{1-m}b^{1-n}(Y+Z)^mY^{n-1} \\
  &\ \ \ \ +\sum_{m,n=0}^{\infty}\frac{1}{m!n!}
    S_{m,n}^\tau\begin{pmatrix}\bb&\bc&\ba\\ \by&\bz&\bx\end{pmatrix}
    (-1)b^{1-m}c^{1-n}Y^{m-1}Z^{n-1}(Y+Z) \\
  &\ \ \ \ +\sum_{m,n=0}^{\infty}\frac{1}{m!n!}
    S_{m,n}^\tau\begin{pmatrix}\bc&\ba&\bb\\ \bz&\bx&\by\end{pmatrix}
    (-1)^nc^{1-m}a^{1-n}Z^{m-1}(Y+Z)^n \\
  &\ \ \ =0.
\end{split}
\end{equation}
Now, taking the coefficients of
$Y^2Z^{-1}$, $Z^2Y^{-1}$ and $Y$
in \eqref{eqn3.5},
we obtain the identities
\eqref{eqn3.2}, \eqref{eqn3.3} and \eqref{eqn3.4},
respectively.
\end{proof}

Now we are ready to prove Proposition \ref{prop3.1}.
\begin{proof}[The first proof of Proposition \ref{prop3.1}]
>From the three identities
\eqref{eqn3.2},\eqref{eqn3.3} and \eqref{eqn3.4}
we have
\begin{equation}\label{eqn3.6}
\begin{split}
\frac{a}{2b}S_{0,2}^\tau&\begin{pmatrix}\ba&\bb&\bc\\ \bx&\by&\bz\end{pmatrix}
-S_{1,1}^\tau\begin{pmatrix}\ba&\bb&\bc\\ \bx&\by&\bz\end{pmatrix}
+\frac{b}{2a}\frac{a}{c}S_{0,2}^\tau\begin{pmatrix}\bb&\bc&\ba\\ \by&\bz&\bx\end{pmatrix}
\\
&-S_{1,1}^\tau\begin{pmatrix}\bb&\bc&\ba\\ \by&\bz&\bx\end{pmatrix}
-\frac{c}{2b}S_{2,0}^\tau\begin{pmatrix}\bb&\bc&\ba\\ \by&\bz&\bx\end{pmatrix}
\\
&+\frac{c}{a}\frac{a}{b}S_{2,0}^\tau\begin{pmatrix}\bb&\bc&\ba\\ \by&\bz&\bx\end{pmatrix}
-S_{1,1}^\tau\begin{pmatrix}\bc&\ba&\bb\\ \bz&\bx&\by\end{pmatrix}
=0.
\end{split}
\end{equation}
We set
\begin{equation*}
\ba=(1,1),\ \bb=(p,p),\ \bc=(q,q)\ \ \text{and}\ \
\bx=(s,0),\ \by=(pt,0),\ \bz=(-qt,0).
\end{equation*}
Note that the conditions \eqref{eqn2.3} are satisfied in this setting.
>From \eqref{eqn3.6} and \eqref{eqn2.2} we have
{\allowdisplaybreaks
\begin{equation}\label{eqn3.7}
\begin{split}
0=&\frac{a}{2b}S_{0,2}^\tau\begin{pmatrix}\ba&\bb&\bc\\ \bx&\by&\bz\end{pmatrix}
-S_{1,1}^\tau\begin{pmatrix}\ba&\bb&\bc\\ \bx&\by&\bz\end{pmatrix}
+\frac{b}{2c}S_{0,2}^\tau\begin{pmatrix}\bb&\bc&\ba\\ \by&\bz&\bx\end{pmatrix}
\\
&-S_{1,1}^\tau\begin{pmatrix}\bb&\bc&\ba\\ \by&\bz&\bx\end{pmatrix}
+\frac{c}{2b}S_{2,0}^\tau\begin{pmatrix}\bb&\bc&\ba\\ \by&\bz&\bx\end{pmatrix}
-S_{1,1}^\tau\begin{pmatrix}\bc&\ba&\bb\\ \bz&\bx&\by\end{pmatrix} \\
=&\frac{1}{2pq}
  \sum_{\substack{\mu\pmod q \\ \lambda\pmod {q}}}
  B_2\left(\frac{p\lambda}{q}-2pt,\frac{p\mu}{q};\tau\right)\\
 &-\frac{1}{q}\sum_{\substack{\mu\pmod q \\ \lambda\pmod {q}}}
  B_1\left(\frac{\lambda}{q}-t-s,\frac{\mu}{q};\tau\right)
  B_1\left(\frac{p\lambda}{q}-2pt,\frac{p\mu}{q};\tau\right)\\
 &+\frac{p}{2q}B_2\left(qs+qt,0;\tau\right)
  -B_1\left(ps-pt,0;\tau\right)B_1\left(qs+qt,0;\tau\right) \\
 &+\frac{q}{2p}B_2\left(ps-pt,0;\tau\right)\\
 &-\frac{1}{p}\sum_{\substack{\mu\pmod p \\ \lambda\pmod {p}}}
  B_1\left(\frac{q\lambda}{p}+2qt,\frac{q\mu}{p};\tau\right)
  B_1\left(\frac{\lambda}{p}+t-s,\frac{\mu}{p};\tau\right).\\
\end{split}
\end{equation}
}
Now we will take the limit of the last expression in \eqref{eqn3.7}
as $t$ tends to $0$.
Extra care should be taken for the terms involving $\lambda=\mu=0$,
as a priori, $B_1(0,0;\tau)$ is not defined.
To go around this difficulty, we will make use of the following expansion of
$B_1(x,0;\tau)$ at $x=0$ (this will be proved later in Lemma \ref{lem4.1}):
\begin{equation*}
  B_1\left(x,0;\tau\right)
  =-\frac{1}{2\pi i}
  \left[\frac{1}{x}
  -E_2(\tau)x-E_4(\tau)x^3-\cdots
  \right].
\end{equation*}
We have
\begin{equation*}
\begin{split}
 -\frac{1}{q}
  &B_1\left(-t-s,0;\tau\right)
  B_1\left(-2pt,0;\tau\right)
  =
 -\frac{1}{q}
  B_1\left(t+s,0;\tau\right)
  B_1\left(2pt,0;\tau\right)\\
  &=\frac{1}{2\pi iq}
    \left[B_1\left(s,0;\tau\right)
    +\frac{\partial B_1\left(s,0;\tau\right)}
    {\partial s}t+\cdots\right]
    \left[\frac{1}{2pt}-E_2(\tau)(2pt)-\cdots\right]
\end{split}
\end{equation*}
and
\begin{equation*}
\begin{split}
 -\frac{1}{p}
  &B_1\left(2qt,0;\tau\right)
  B_1\left(t-s,0;\tau\right)
  =
  \frac{1}{p}
  B_1\left(2qt,0;\tau\right)
  B_1\left(-t+s,0;\tau\right)\\
  &=-\frac{1}{2\pi ip}
    \left[\frac{1}{2qt}-E_2(\tau)(2qt)-\cdots\right]
    \left[B_1\left(s,0;\tau\right)
    +\frac{\partial B_1\left(s,0;\tau\right)}
    {\partial s}(-t)+\cdots\right].
\end{split}
\end{equation*}
Hence we know
\begin{equation*}
\begin{split}
\lim_{t\to 0}
  &\left[ -\frac{1}{q}
  B_1\left(-t-s,0;\tau\right)
  B_1\left(-2pt,0;\tau\right)
   -\frac{1}{p}
  B_1\left(2qt,0;\tau\right)
  B_1\left(t-s,0;\tau\right)
  \right] \\
  &=\frac{1}{2\pi ipq}\frac{\partial B_1\left(s,0;\tau\right)}{\partial s}.
\end{split}
\end{equation*}
>From this we know that the last expression in \eqref{eqn3.7}
converges to
\begin{equation*}
\begin{split}
  &\frac{1}{2pq}
  \sum_{\substack{\mu\pmod q \\ \lambda\pmod {q}}}
  B_2\left(\frac{p\lambda}{q},\frac{p\mu}{q};\tau\right)\\
 &-\frac{1}{q}\sum_{\substack{\lambda,\mu=0\\ (\lambda,\mu)\ne (0,0)}}
  B_1\left(\frac{\lambda}{q}-s,\frac{\mu}{q};\tau\right)
  B_1\left(\frac{p\lambda}{q},\frac{p\mu}{q};\tau\right)\\
 &+\frac{p}{2q}B_2\left(qs,0;\tau\right)
  -B_1\left(ps,0;\tau\right)
  B_1\left(qs,0;\tau\right)
  +\frac{q}{2p}
  B_2\left(ps,0;\tau\right)\\
 &-\frac{1}{p}\sum_{\substack{\lambda,\mu=0\\ (\lambda,\mu)\ne (0,0)}}
  B_1\left(\frac{q\lambda}{p},\frac{q\mu}{p};\tau\right)
  B_1\left(\frac{\lambda}{p}-s,\frac{\mu}{p};\tau\right)\\
 &+\frac{1}{2\pi ipq}\frac{\partial B_1\left(s,0;\tau\right)}{\partial s}
\end{split}
\end{equation*}
when $t$ tends to $0$.

Finally, setting
\begin{equation*}
C(\tau)=\frac{1}{2pq}
  \sum_{\substack{\mu\pmod q \\ \lambda\pmod {q}}}
  B_2\left(\frac{p\lambda}{q},\frac{p\mu}{q};\tau\right),\\
\end{equation*}
we obtain the identity \eqref{eqn3.1}. This completes the proof.
\end{proof}

Now we will give the second proof, which was kindly
communicated to us by the referee.

\begin{proof}[The second proof of Proposition \ref{prop3.1}]
We recall the identity \eqref{eqn2.0}
\begin{equation*}
  \underline{F}(x,y;\xi;\tau)
  =\sum_{m=0}^{\infty}\frac{B_m(x,y;\tau)}{m!}(2\pi i)^m\xi^{m-1}.
\end{equation*}
According to a classical result of Kronecker (refer to Weil \cite{WE1}),
the left hand side above admits the following partial fraction
decomposition
\begin{equation}\label{eqn3.8}
  \underline{F}(x,y;\xi;\tau)
  =\lim_{M \to \infty}\sum_{m=-M}^{M}
  \left(\lim_{N \to \infty}\sum_{n=-N}^{N}
  \frac{\chi(w\bar{z})}{w+\xi}
  \right)
\end{equation}
where $z=-x+y\tau$, $w=m\tau+n$ and
$\chi(t)=\exp(2\pi i\Im(t)/\Im(\tau))$.
Expanding the right hand side of \eqref{eqn3.8} into a power series in $\xi$,
we have
\begin{equation}\label{eqn3.9}
  \underline{F}(x,y;\xi;\tau)
  =\sum_{k=0}^{\infty}(-1)^{k-1}C_k(z)\xi^{k-1}
\end{equation}
where
\begin{equation}\label{eqn3.10}
  C_k(z)
  =\lim_{M \to \infty}\sum_{m=-M}^{M}
  \left(\lim_{N \to \infty}\sum_{n=-N}^{N}
  \frac{\chi(w\bar{z})}{w^k}
  \right).
\end{equation}

Therefore, from \eqref{eqn2.0} and \eqref{eqn3.9}, we know
\begin{equation}\label{eqn3.11}
  B_k(x,y;\tau)=\frac{(-1)^{k-1}k!}{(2\pi i)^k}C_k(-x+y\tau).
\end{equation}

In what follows, we use the notation $\EE(z)$ and $\EE_k(z)$
in place of $E(z)$ and $E_k(z)$ in Sczech \cite{SC1}
to distinguish them from the Eisenstein series. First we note that
\begin{equation}\label{eqn3.11.1}
  C_1(z)=\EE_1(z),\ \ \ C_2(z)=\EE(z).
\end{equation}

Now we apply Satz 1 in Sczech \cite[p.\ 530]{SC1}, setting
$$c_1=p,\ c_2=q,\ c_3=1,\ z_1=z_2=0,\ z_3=x.$$
This gives the following reciprocity law
\begin{equation}\label{eqn3.12}
\end{equation}
\begin{align*}
  \frac{q}{p}\EE(px)+\sum_{k=0,1}&q^{1-k}p^{k-1}
    \sum_{r\in L/L}\EE_k(px)\EE_{2-k}(qx) \\
  +\frac{1}{pq}\EE(0)&+\sum_{k=0,1}\frac{q^{k-1}}{p}
    \sum_{r\in L/pL}\EE_k(\frac{rq}{p})\EE_{2-k}(\frac{r}{p}-x) \\
  &+\frac{p}{q}\EE(-qx)+\sum_{k=0,1}\frac{p^{1-k}}{q}
    \sum_{r\in L/qL}\EE_k(\frac{r}{q}-x)\EE_{2-k}(\frac{rp}{q})=0
\end{align*}
where $L$ denotes the lattice $\ZZ\tau+\ZZ$.

Furthermore, it was shown in \cite{SC1} (using results of Hecke)
that
\begin{equation}\label{eqn3.12.1}
\end{equation}
\begin{align*}
  \EE_1(z)&=\zeta(z;\tau)-zE_2(\tau)+2\pi iy, \\
  2\EE(z)&=\wp(z;\tau)-C_1(z)^2, \\
  \EE(0)&=E_2(\tau)-\frac{\pi i}{\Im(\tau)}, \\
  \EE_0(z)&=\begin{cases}
         -1 & z\in L \\
          0 & \text{otherwise},
         \end{cases}
\end{align*}
where $\wp(z;\tau)$ and $\zeta(z;\tau)$ are the Weierstrass pe and zeta
functions and $E_2(\tau)$ is
the Eisenstein series of weight two.

Now we take $x$ to be a sufficiently small and $x\ne 0$ so that
$px$ and $qx$ are not rational integers.
Then the equation \eqref{eqn3.12} combined with the identities
\eqref{eqn3.11.1} and \eqref{eqn3.12.1}
produces the following formula
\begin{align*}
  (2\pi i)^2[D^-(p,q;\tau;x)&+D^-(q,p;\tau;x)] \\
  &=-C_1(qx)C_1(px)+\frac{\wp(x)}{pq}-\frac{p}{q}C_2(qx)-\frac{q}{p}C_2(px).
\end{align*}
This formula and the identities \eqref{eqn3.11} imply
Proposition \ref{prop3.1}.
\end{proof}

\section{Weierstrass elliptic functions
and elliptic Bernoulli functions}\label{sect4}

In this section we study the relationship between the Weierstrass
elliptic functions and the elliptic Bernoulli functions.

For $z\in\CC$ and $\tau\in\HH$,
the Weierstrass sigma, zeta and pe functions are given as follows:
\begin{align*}
\sigma(z;\tau)&:=z\prod
        _{\substack{\gamma\in \ZZ\tau+\ZZ \\ \gamma\ne 0}}
              \left(1-\frac{z}{\gamma}\right)
              \exp\left(\frac{z}{\gamma}+\frac{1}{2}
                \left(\frac{z}{\gamma}\right)^2\right), \\
\zeta(z;\tau)&:=\frac{1}{z}+\sum
        _{\substack{\gamma\in \ZZ\tau+\ZZ \\ \gamma\ne 0}}
              \left(\frac{1}{z-\gamma}+\frac{1}{\gamma}
              +\frac{z}{\gamma^2}\right), \\
\PP(z;\tau)&:=\frac{1}{z^2}+\sum
        _{\substack{\gamma\in \ZZ\tau+\ZZ \\ \gamma\ne 0}}
              \left(\frac{1}{(z-\gamma)^2}-\frac{1}{\gamma^2}\right).
\end{align*}
It is known that these functions have the following expansions
at $z=0$:
\begin{align*}
\log\sigma(z;\tau)&=\log z
        -\sum_{n=2}^{\infty}\frac{1}{2n}E_{2n}(\tau)z^{2n}, \\
\zeta(z;\tau)&=\frac{1}{z}
        -\sum_{n=2}^{\infty}E_{2n}(\tau)z^{2n-1}, \\
\PP(z;\tau)&=\frac{1}{z^2}
        +\sum_{n=2}^{\infty}(2n-1)E_{2n}(\tau)z^{2n-2}
\end{align*}
where $E_{2n}(\tau)$ is the Eisenstein series of weight $2n$, namely,
  $$E_{2n}(\tau):=\sum
        _{\substack{\gamma\in \ZZ\tau+\ZZ \\ \gamma\ne 0}}
              \frac{1}{\gamma^{2n}}. $$
It is also known that $E_{2n}(\tau)$ have the following expansion:
\begin{equation*}
\begin{split}
E_{2n}(\tau)
      &=2\zeta(2n)
        +\frac{2(2\pi i)^{2n}}{(2n-1)!}
        \sum_{j=1}^{\infty}\sum_{k=1}^{\infty}k^{2n-1}q^{kj}\\
      &=2\zeta(2n)
        +\frac{2(2\pi i)^{2n}}{(2n-1)!}\sum_{k=1}^{\infty}k^{2n-1}
        \frac{q^k}{1-q^k} \\
      &=2\zeta(2n)
        +\frac{2(2\pi i)^{2n}}{(2n-1)!}\sum_{k=1}^{\infty}\sigma_{2n-1}(k)q^k
        \ \ (\sigma_{\ell}(k)=\sum_{d|k}d^{\ell})
\end{split}
\end{equation*}
where $\zeta(z)$ denotes the Riemann zeta function, and it holds that
\begin{equation*}
        2\zeta(2n)=-\frac{(2\pi i)^{2n}B_{2n}}{(2n)!}. \\
\end{equation*}

Now it is easy to see that these functions
have the following relation:
\begin{equation*}
\frac{\partial\log\sigma(z;\tau)}{\partial z}=\zeta(z;\tau),\ \
\frac{\partial\,\zeta(z;\tau)}{\partial z}=-\PP(z;\tau).
\end{equation*}
The function $\zeta(z;\tau)$ is subject to the following identities
(\cite[p.\ 84]{WA1}):
\begin{equation}\label{eqn4.1}
  \zeta(z+1;\tau)=\zeta(z;\tau)+E_{2}(\tau),\ \
  \zeta(z+\tau;\tau)=\zeta(z;\tau)+E_{2}(\tau)\tau-2\pi i.
\end{equation}

Next we express the elliptic Bernoulli functions of lower degrees
in terms of the Weierstrass elliptic functions and the
Eisenstein series.

\begin{lem}\label{lem4.1}
For sufficiently small real numbers $x$, $y$ and $\tau\in\HH$,
it holds that
\begin{align}
B_1(x,y;\tau) \label{eqn4.2}
  &=-\frac{1}{2\pi i}\left[
   \frac{1}{x-y\tau}
    -\sum_{n=1}^{\infty}E_{2n}(\tau)(x-y\tau)^{2n-1}
    \right]+y \\ \label{eqn4.3}
  &=-\frac{1}{2\pi i}\left[
    \zeta(x-y\tau;\tau)-E_2(\tau)(x-y\tau)
    \right]+y, \\ \label{eqn4.4}
B_2(x,0;\tau)
  &=-\frac{1}{\pi i}
    \left[\sum_{n=1}^{\infty}
    \frac{1}{2n}
    \frac{\partial E_{2n}(\tau)}
    {\partial \tau}x^{2n}
    +\frac{1}{2\pi i}E_2(\tau)
    \right] \\ \label{eqn4.5}
  &=\frac{1}{2\pi i}
    \left[2\frac{\partial\log\sigma(x;\tau)}{\partial \tau}
    -\frac{\partial E_2(\tau)}{\partial \tau}x^2
    -\frac{1}{\pi i}E_2(\tau)
    \right], \\ \label{eqn4.6}
\frac{\partial B_1(x,0;\tau)}{\partial x}
  &=\frac{1}{2\pi i}\left[
    \frac{1}{x^2}
        +\sum_{n=1}^{\infty}(2n-1)E_{2n}(\tau)x^{2n-2}
    \right] \\ \label{eqn4.7}
  &=\frac{1}{2\pi i}\left[
    \PP(x;\tau)+E_2(\tau)
    \right].
\end{align}
\end{lem}
\begin{proof}
The proof is based upon the expansion \eqref{eqn2.1}
and direct calculations:
{\allowdisplaybreaks
\begin{align*}
B_1&(x,y;\tau) \\
  &=
  \sum_{j=1}^{\infty}\frac{\be(-y\tau)q^j}{\be(-x)-\be(-y\tau)q^j}
  -\sum_{j=1}^{\infty}\frac{\be(y\tau)q^j}{\be(x)-\be(y\tau)q^j}
  +\frac{\be(-x+y\tau)}{\be(-x+y\tau)-1}
  +B_1(y) \\
  &=\sum_{j=1}^{\infty}\sum_{k=1}^{\infty}
    \left[\be(k(x-y\tau))q^{kj}-\be(-k(x-y\tau))q^{kj}\right] \\
  &\ccccsp  -\frac{1}{2\pi i(x-y\tau)}\frac{2\pi i(x-y\tau)}{\be(x-y\tau)-1}+B_1+y \\
  &=2\sum_{\substack{n=0 \\ n\ \text{odd}}}^{\infty}
    \frac{1}{n!}\left[\sum_{j=1}^{\infty}\sum_{k=1}^{\infty}
    k^nq^{kj}\right](2\pi i(x-y\tau))^n
    -\sum_{\substack{n=0 \\ n\ \text{even}}}^{\infty}
    \frac{B_n}{n!}(2\pi i(x-y\tau))^{n-1}+y \\
  &=-\frac{1}{2\pi i(x-y\tau)} \\
    &\csp\ \ \ +\sum_{n=1}^{\infty}
    \left[-\frac{B_{2n}}{(2n)!}+\frac{2}{(2n-1)!}
      \sum_{j=1}^{\infty}\sum_{k=1}^{\infty}
    k^{2n-1}q^{kj}\right](2\pi i(x-y\tau))^{2n-1}+y \\
  &=-\frac{1}{2\pi i(x-y\tau)}
        +\frac{1}{2\pi i}\sum_{n=1}^{\infty}E_{2n}(\tau)(x-y\tau)^{2n-1}+y \\
  &=-\frac{1}{2\pi i}\left[\zeta(x-y\tau;\tau)-E_2(\tau)(x-y\tau)\right]+y,
\end{align*}
\begin{align*}
B_2(x,0;\tau)
  &=2\left[
  \sum_{j=1}^{\infty}(-j)\frac{q^j}{\be(-x)-q^j}
  -\sum_{j=1}^{\infty}j\frac{q^j}{\be(x)-q^j}
  \right]+B_2 \\
  &=-2\sum_{j=1}^{\infty}\sum_{k=1}^{\infty}
    j\left[\be(kx)q^{kj}+\be(-kx)q^{kj}\right]
    +B_2 \\
  &=-4\sum_{\substack{n=0 \\ n\ \text{even}}}^{\infty}
    \frac{1}{n!}\left[\sum_{j=1}^{\infty}\sum_{k=1}^{\infty}
    jk^nq^{kj}\right](2\pi ix)^n+B_2 \\
  &=-4\sum_{n=1}^{\infty}
    \frac{1}{(2n)!}\left[\sum_{j=1}^{\infty}\sum_{k=1}^{\infty}
    jk^{2n}q^{kj}\right](2\pi ix)^{2n}
    -4\sum_{j=1}^{\infty}\frac{jq^j}{1-q^j}+B_2 \\
  &=-\frac{1}{\pi i}\sum_{n=1}^{\infty}
    \frac{1}{2n}
    \frac{\partial E_{2n}(\tau)}
    {\partial \tau}x^{2n}
    -\frac{2}{(2\pi i)^2}E_2(\tau) \\
  &=\frac{1}{2\pi i}
    \left[2\frac{\partial\log\sigma(x;\tau)}{\partial \tau}
    -\frac{\partial E_2(\tau)}{\partial \tau}x^2
    -\frac{1}{\pi i}E_2(\tau)
    \right], \\
\frac{\partial B_1(x,0;\tau)}{\partial x}
  &=\frac{1}{2\pi i}\left[
    \frac{1}{x^2}
        +\sum_{n=1}^{\infty}(2n-1)E_{2n}(\tau)x^{2n-2}
    \right] \\
  &=\frac{1}{2\pi i}\left[
    \PP(x;\tau)+E_2(\tau)
    \right].
\end{align*}
}
These give the identities from \eqref{eqn4.2} to \eqref{eqn4.7}.
\end{proof}

\section{Proofs of Theorems \ref{thm1.1} and \ref{thm1.3}}
\label{sect5}
In this section we give proofs of Theorems \ref{thm1.1} and \ref{thm1.3}.

\begin{proof}[Proof of Theorem \ref{thm1.3}]
For $\lambda,\mu$ such that $(\lambda,\mu)\ne (0,0)$, the identity
\eqref{eqn4.3} gives
\begin{align*}
  B_1\left(\frac{\lambda}{p}-x,-\frac{\mu}{p};\tau\right)
  &=-\frac{1}{2\pi i}\left[
    \zeta\left(
    \frac{\lambda+\mu\tau}{p}-x;\tau\right)
  -E_2(\tau)\left(\frac{\lambda+\mu\tau}{p}-x\right)
    +2\pi i\frac{\mu}{p}\right], \\
  B_1\left(\frac{q\lambda}{p},-\frac{q\mu}{p};\tau\right)
  &=-\frac{1}{2\pi i}\left[
    \zeta\left(
    \frac{q(\lambda+\mu\tau)}{p};\tau\right)
  -E_2(\tau)\frac{q(\lambda+\mu\tau)}{p}
    +2\pi i\frac{q\mu}{p}\right].
\end{align*}
Hence we have
\begin{equation}\label{eqn5.1}
\end{equation}
{\allowdisplaybreaks
\begin{align*}
    \frac{1}{p}&\sum_{\substack{\lambda,\mu=0\\ (\lambda,\mu)\ne (0,0)}}^{p-1}
    B_1\left(\frac{\lambda}{p}-x,\frac{\mu}{p};\tau\right)
    B_1\left(\frac{q\lambda}{p},\frac{q\mu}{p};\tau\right) \\
  &=\frac{1}{p}\sum_{\substack{\lambda,\mu=0\\ (\lambda,\mu)\ne (0,0)}}^{p-1}
    B_1\left(\frac{\lambda}{p}-x,-\frac{\mu}{p};\tau\right)
    B_1\left(\frac{q\lambda}{p},-\frac{q\mu}{p};\tau\right) \\
  &\ccsp\text{(since $B_1(x,y+1;\tau)=B_1(x,y;\tau)$)} \\
  &=\frac{1}{(2\pi i)^2p}
    \sum_{\substack{\lambda,\mu=0\\ (\lambda,\mu)\ne (0,0)}}^{p-1}
  \left[
    \zeta\left(
    \frac{\lambda+\mu\tau}{p}-s;\tau\right)
  -E_2(\tau)\left(\frac{\lambda+\mu\tau}{p}-s\right)
    +2\pi i\frac{\mu}{p}\right] \\
  &\ccsp\ccsp \times\left[
    \zeta\left(
    \frac{q(\lambda+\mu\tau)}{p};\tau\right)
  -E_2(\tau)\frac{q(\lambda+\mu\tau)}{p}
    +2\pi i\frac{q\mu}{p}\right] \\
  &=D^-(p,q;\tau;x).
\end{align*}
}

Moreover the identities \eqref{eqn4.3}, \eqref{eqn4.5}
and \eqref{eqn4.7} give us
\begin{equation}\label{eqn5.2}
\begin{split}
  -B_1(px,0;\tau)&B_1(qx,0;\tau)
    +\frac{q}{2p}B_2(px,0;\tau)
    +\frac{p}{2q}B_2(qx,0;\tau)
  +\frac{1}{2\pi ipq}\frac{\partial B_1(x,0;\tau)}{\partial x} \\
  =&-\frac{1}{(2\pi i)^2}
     \left[
      \zeta(px;\tau)-E_2(\tau)(px)
     \right]
     \left[
      \zeta(qx;\tau)-E_2(\tau)(qx)
     \right] \\
  &+\frac{q}{4\pi ip}
  \left[
  2\frac{\partial\log\sigma(px;\tau)}{\partial \tau}
  -\frac{\partial E_2(\tau)}{\partial \tau}(px)^2
  -\frac{1}{\pi i}E_2(\tau)
  \right] \\
  &+\frac{p}{4\pi iq}
  \left[
  2\frac{\partial\log\sigma(qx;\tau)}{\partial \tau}
  -\frac{\partial E_2(\tau)}{\partial \tau}(qx)^2
  -\frac{1}{\pi i}E_2(\tau)
  \right] \\
  &+\frac{1}{2\pi ipq}\frac{1}{2\pi i}
  \left[
  \PP(x;\tau)+E_2(\tau)
  \right] \\
  =&R^-(p,q;\tau;x).
\end{split}
\end{equation}

Now the identities \eqref{eqn5.1} and \eqref{eqn5.2}
together with Proposition \ref{prop3.1} give (2) of Theorem \ref{thm1.3}.

The assertion (1) follows easily from \eqref{eqn4.1} and the fact
that $\zeta(z;\tau)$ is an odd function with respect to $z$.
This completes the proof.
\end{proof}

Next we give a proof of Theorem \ref{thm1.1}.
\begin{proof}[Proof of Theorem \ref{thm1.1}]
Let us consider the Taylor expansion of
\begin{align*}
  D^-(p,q;&\tau;x) \\
  =&\frac{1}{(2\pi i)^2p}
    \sum_{\substack{\lambda,\mu=0\\ (\lambda,\mu)\ne (0,0)}}^{p-1}
  \left[
    \zeta\left(
    \frac{\lambda+\mu\tau}{p}-x;\tau\right)
  -E_2(\tau)\left(\frac{\lambda+\mu\tau}{p}-x\right)
    +2\pi i\frac{\mu}{p}\right] \\
  &\ccsp\ccsp \times\left[
    \zeta\left(
    \frac{q(\lambda+\mu\tau)}{p};\tau\right)
  -E_2(\tau)\frac{q(\lambda+\mu\tau)}{p}
    +2\pi i\frac{q\mu}{p}\right]
\end{align*}
at $x=0$.
Then we see that the coefficient of $x^{2n}$ in this expansion is
equal to
\begin{align*}
  &\frac{1}{(2\pi i)^2p(2n)!}
    \sum_{\substack{\lambda,\mu=0\\ (\lambda,\mu)\ne (0,0)}}^{p-1}
    \zeta^{(2n)}\left(
    \frac{\lambda+\mu\tau}{p};\tau\right) \\
  &\cccsp\ccsp\times\left[
    \zeta\left(
    \frac{q(\lambda+\mu\tau)}{p};\tau\right)
  -E_2(\tau)\frac{q(\lambda+\mu\tau)}{p}
    +2\pi i\frac{q\mu}{p}\right].
\end{align*}
This is nothing but $D_{2n}^-(p,q;\tau)$.

Next, applying \eqref{eqn4.2}, \eqref{eqn4.4} and \eqref{eqn4.6},
we expand $R^-(p,q;\tau;x)$ as follows:
{\allowdisplaybreaks
\begin{align*}
  R^-(p,q;\tau;x)
  =&-\frac{1}{(2\pi i)^2}
     \left[
      \zeta(px;\tau)-E_2(\tau)(px)
     \right]
     \left[
      \zeta(qx;\tau)-E_2(\tau)(qx)
     \right] \\
  &+\frac{q}{4\pi ip}
  \left[
  2\frac{\partial\log\sigma(px;\tau)}{\partial \tau}
  -\frac{\partial E_2(\tau)}{\partial \tau}(px)^2
  -\frac{1}{\pi i}E_2(\tau)
  \right] \\
  &+\frac{p}{4\pi iq}
  \left[
  2\frac{\partial\log\sigma(qx;\tau)}{\partial \tau}
  -\frac{\partial E_2(\tau)}{\partial \tau}(qx)^2
  -\frac{1}{\pi i}E_2(\tau)
  \right] \\
  &+\frac{1}{2\pi ipq}\frac{1}{2\pi i}
  \left[
  \PP(x;\tau)+E_2(\tau)
  \right] \\
  =&-\frac{1}{(2\pi i)^2}\left[
        \frac{1}{px}
        -\sum_{n=1}^{\infty}E_{2n}(\tau)(px)^{2n-1}
     \right]
     \left[
        \frac{1}{qx}
        -\sum_{n=1}^{\infty}E_{2n}(\tau)(qx)^{2n-1}
     \right] \\
  &-\frac{q}{2\pi ip}
  \left[
    \sum_{n=1}^{\infty}\frac{1}{2n}
    \frac{\partial E_{2n}(\tau)}{\partial \tau}(px)^{2n}
    +\frac{1}{2\pi i}E_2(\tau)
  \right] \\
  &-\frac{p}{2\pi iq}
  \left[
    \sum_{n=1}^{\infty}\frac{1}{2n}
    \frac{\partial E_{2n}(\tau)}{\partial \tau}(qx)^{2n}
    +\frac{1}{2\pi i}E_2(\tau)
  \right] \\
  &+\frac{1}{2\pi ipq}\frac{1}{2\pi i}
  \left[
    \frac{1}{x^2}+\sum_{n=1}^{\infty}(2n-1)
    E_{2n}(\tau)x^{2n-2}
  \right].
\end{align*}
}
Thus we know that the coefficient of $x^{2n}$ in this expansion is
equal to
\begin{align*}
  &-\frac{1}{(2\pi i)^2}\Bigg[
        \sum_{j=1}^{n}E_{2j}(\tau)E_{2n+2-2j}(\tau)
        p^{2j-1}q^{2n+1-2j}
        -E_{2n+2}(\tau)\frac{p^{2n+2}+q^{2n+2}}{pq}
     \Bigg] \\
  &-\frac{q}{2\pi ip}
    \frac{1}{2n}\frac{\partial E_{2n}(\tau)}{\partial \tau}p^{2n}
    -\frac{p}{2\pi iq}
    \frac{1}{2n}\frac{\partial E_{2n}(\tau)}{\partial \tau}q^{2n}
   +\frac{1}{(2\pi i)^2}(2n+1)E_{2n+2}(\tau)\frac{1}{pq}.
\end{align*}
This is nothing but $R_{2n}^-(p,q;\tau)$.

Hence, from the reciprocity laws (2) in Theorem \ref{thm1.3},
we obtain the reciprocity laws (2) in Theorem \ref{thm1.1}.

The assertion (1) easily follows from that of Theorem \ref{thm1.3}.
This completes the proof.
\end{proof}

\begin{rem}\label{rem5.1}
{\rm
A direct calculation shows that
\begin{equation*}
  \lim_{\tau\to i\infty}R_{w}^-(p,q;\tau)=
  \frac{2(2\pi i)^{w}}{w!}g_w(p,q).
\end{equation*}
>From this and \eqref{eqn1.5}, we obtain \eqref{eqn1.6}
which shows that $D_{w}^-(p,q;\tau)$ is an elliptic analogue
of the Apostol-Dedekind sum.
}
\end{rem}

\section{A proof of Theorem \ref{thm1.2}}\label{sect6}

In this section we give a proof of Theorem \ref{thm1.2}.
We first set up some notation.
Let $f$ be an element of $S_{w+2}$.
Then $n$th period of $f$, \ $r_{n}(f)$,\  is defined by
\begin{equation*}
  r_{n}(f):=\int_{0}^{i\infty}f(z)z^{n}dz
             \ \ (n=0,1,\ldots,w).
\end{equation*}
Furthermore, the period polynomial $r(f)$
and the odd period polynomial $r^-(f)$ of $f$ in the variables $p$ and $q$
is defined by
  $$r(f)(p,q):=\int_{0}^{i\infty}f(z)(pz-q)^{w}dz \text{\ \ \ and\ \ \ }
  r^-(f)(p,q):=\frac12[r(f)(p,q)-r(f)(p,-q)].$$
It is clear that $r^-(f)(p,q)$ has the following expression:
  $$r^-(f)(p,q)=-\sum_{\substack{n=0 \\ \text{$n$ odd}}}^{w}
    \binom{w}{n}r_{w-n}(f)p^{w-n}q^{n}.$$
Here and hereafter $\binom{w}{n}$ denotes a binomial coefficient.

Let $G_{2n}$ be
a normalized Eisenstein series:
\begin{equation*}
  G_{2n}(\tau)
    :=-\frac{B_{2n}}{4n}
      +\sum_{k=1}^{\infty}\sigma_{2n-1}(k)q^k.
\end{equation*}
Notice that
  $$E_{2n}=\frac{2(2\pi i)^{2n}}{(2n-1)!}G_{2n}.$$

To prove Theorem \ref{thm1.2}, we need the following lemma.
\begin{lem}\label{lem6.1}
Set $w=2n$ and let $\{f_j\}_{j=1}^{d_w}$ be a basis
of normalized eigenforms of $S_{w+2}$.
Then it holds that
\begin{equation}\label{eqn6.1}
\begin{split}
\sum_{j=1}^{n}&\left[G_{2j}G_{2n+2-2j}
  +(\delta_{j,1}+\delta_{j,n})
    \frac{1}{8\pi in}\frac{\partial G_{2n}}{\partial\tau}
  +\frac{1}{2}\frac{B_{2j}}{2j}\frac{B_{2n+2-2j}}{2n+2-2j}
    \frac{2n+2}{B_{2n+2}}
  G_{2n+2}\right] \\
  &\ccccsp\times\binom{2n}{2j-1}p^{2j-1}q^{2n+1-2j}\\
  &\ccsp=-\frac{1}{(2i)^{2n+1}}
    \sum_{j=1}^{d_w}\frac{r_{2n}(f_j)r^{-}(f_j)(p,q)}{(f_j,f_j)}f_j
\end{split}
\end{equation}
where  $\delta_{i,j}$ is the Kronecker delta symbol,
 and
$(f,g)$ denotes the Petersson inner product of $f$ and $g$.
\end{lem}
\begin{proof}
We use the following Rankin's identity
(refer to Kohnen-Zagier \cite{KZ1} noting that their notation of $r_n(f)$
differs from ours by a factor $i^{n+1}$):
for a normalized eigenform $f$ of $S_{w+2}$,
\begin{equation}\label{eqn6.2}
\left(f,G_{2j}G_{2n+2-2j}
  +(\delta_{j,1}+\delta_{j,n})
    \frac{1}{8\pi in}\frac{\partial G_{2n}}{\partial\tau}\right)
=\frac{1}{(2i)^{2n+1}}r_{2n}(f)r_{2j-1}(f)
\end{equation}
where $j=1,2,\ldots,n$.

We set $g$ and $h$ to be the left and right hand sides of \eqref{eqn6.1},
respectively. We note that both $g$ and $h$ are cusp forms of $S_{w+2}$.
Then, for any $f_l\ (l=1,2,\ldots,d_w)$, we have
{\allowdisplaybreaks
\begin{align*}
(f_l,g)
&=\sum_{j=1}^{n}\left(f_l,G_{2j}G_{2n+2-2j}
  +(\delta_{j,1}+\delta_{j,n})
    \frac{1}{8\pi in}\frac{\partial G_{2n}}{\partial\tau}
    \right)\binom{2n}{2j-1}p^{2j-1}q^{2n+1-2j} \\
&\ccsp\text{(\,since $(f_l,G_{2n+2})=0$\,)} \\
&=\sum_{j=1}^{n}
  \frac{1}{(2i)^{2n+1}}r_{2n}(f_l)r_{2j-1}(f_l)
    \binom{2n}{2j-1}p^{2j-1}q^{2n+1-2j} \\
&\ccsp\text{(\,by \eqref{eqn6.2}\,)} \\
&=-\frac{1}{(2i)^{2n+1}}r_{2n}(f_l)r^-(f_l)(p,q) \\
&=(f_l,h).
\end{align*}
}
This implies \eqref{eqn6.1}.
\end{proof}

Now we are ready to prove Theorem \ref{thm1.2}.
\begin{proof}[Proof of Theorem \ref{thm1.2}]
Set $w=2n$ and let $\{f_j\}_{j=1}^{d_w}$ be a basis
of normalized eigenforms of $S_{w+2}$.

We use the formulas (\cite[pp.\ 453--454]{ZA1})
\begin{align*}
  r_{2n}(G_{2n+2})&=\frac{(2n)!\zeta(2n+1)}{2(2\pi i)^{2n+1}},
  \\
  (G_{2n+2},G_{2n+2})&=\frac{(2n)!}{(4\pi)^{2n+1}}
    \frac{B_{2n+2}}{2(2n+2)}\zeta(2n+1)
\end{align*}
and the formula (\cite{FU2})
  $$r^-(G_{2n+2})(p,q)=
    -\frac{1}{pq}
    \left\{
    \sum^{n+1}_{j=0}
    \frac{(2n)!B_{2j}B_{2n+2-2j}}{2(2j)!(2n+2-2j)!}p^{2j}q^{2n+2-2j}
    +\frac{B_{2n+2}}{4(n+1)}
    \right\},
  $$
to reformulate $R_{2n}^-(p,q;\tau)$ as follows:
\begin{equation}\label{eqn6.3}
\end{equation}
{\allowdisplaybreaks
\begin{align*}
  R_{2n}^-&(p,q;\tau) \\
  =&-\frac{1}{(2\pi i)^2pq}\Bigg[
        \sum_{j=1}^{n}E_{2j}(\tau)E_{2n+2-2j}(\tau)
        p^{2j}q^{2n+2-2j} \\
        &\ccsp\csp-E_{2n+2}(\tau)(p^{2n+2}+q^{2n+2})
        -(2n+1)E_{2n+2}(\tau)
     \Bigg] \\
  &-\frac{1}{4\pi in}\frac{\partial E_{2n}(\tau)}{\partial \tau}
    (p^{2n-1}q+pq^{2n-1}) \\
  =&-\frac{4(2\pi i)^{2n+2}}{(2\pi i)^2(2n)!}
        \sum_{j=1}^{n}
        \left[
        G_{2j}(\tau)G_{2n+2-2j}(\tau)
        +(\delta_{j,1}+\delta_{j,n})
        \frac{1}{8\pi in}
        \frac{\partial G_{2n}(\tau)}{\partial \tau}
        \right]
     \\
  &\cccsp\times\binom{2n}{2j-1}p^{2j-1}q^{2n+1-2j}
     \\
  &+\frac{2(2\pi i)^{2n}}{(2n+1)!}
    G_{2n+2}(\tau)\frac{p^{2n+2}+q^{2n+2}}{pq}
    +\frac{2(2\pi i)^{2n}}{(2n)!}
    G_{2n+2}(\tau)\frac{1}{pq} \\
  =&
  \frac{4(2\pi i)^{2n}}{(2n)!}\frac{1}{(2i)^{2n+1}}
    \sum_{j=1}^{d_w}\frac{r_{2n}(f_j)r^{-}(f_j)(p,q)}{(f_j,f_j)}f_j(\tau)
     \\
  &+
  \frac{2(2\pi i)^{2n}}{(2n)!}
    \frac{2n+2}{B_{2n+2}}(2n)!
  \sum_{j=0}^{n+1}
    \Bigg[
    \frac{B_{2j}}{(2j)!}\frac{B_{2n+2-2j}}{(2n+2-2j)!}
    p^{2j-1}q^{2n+1-2j} \\
    &\ccccsp\csp+\frac{2n+1}{(2n+2)!}
    B_{2n+2}\frac{1}{pq}
    \Bigg]G_{2n+2}(\tau) \\
  &\ccsp\text{(by Lemma \ref{lem6.1})} \\
  =&
  \frac{4(2\pi i)^{2n}}{(2n)!}\frac{1}{(2i)^{2n+1}}
    \sum_{j=1}^{d_w}\frac{r_{2n}(f_j)r^{-}(f_j)(p,q)}{(f_j,f_j)}f_j(\tau)
     \\
  &-
  \frac{4(2\pi i)^{2n}}{(2n)!}
  \frac{2n+2}{B_{2n+2}} r^-(G_{2n+2})(p,q)G_{2n+2}(\tau) \\
  =&
  \frac{4(2\pi i)^{2n}}{(2n)!}\frac{1}{(2i)^{2n+1}}
    \sum_{j=1}^{d_w}\frac{r_{2n}(f_j)r^{-}(f_j)(p,q)}{(f_j,f_j)}f_j(\tau)
     \\
  &+\frac{4(2\pi i)^{2n}}{(2n)!}\frac{1}{(2i)^{2n+1}}
    \frac{r_{2n}(G_{2n+2})r^{-}(G_{2n+2})(p,q)}
    {(G_{2n+2},G_{2n+2})}G_{2n+2}(\tau).
\end{align*}
}
By setting
  $$f_0(\tau):=G_{2n+2}(\tau),
  $$
the identity \eqref{eqn6.3} can be rewritten as
\begin{equation}\label{eqn6.4}
  R_{2n}^-(p,q;\tau)=
  -\frac{2i\pi^{2n}}{(2n)!}
    \sum_{j=0}^{d_w}\frac{r_{2n}(f_j)r^{-}(f_j)(p,q)}{(f_j,f_j)}f_j(\tau).
\end{equation}

Now, since $f_j$ $(j=0,1,\cdots,d_w)$ form a basis of
$M_{w+2}$, thus they are linearly independent. Hence
there are $\tau_0,\tau_1,\ldots,\tau_{d_w}\in\HH$ such that
\begin{equation}\label{eqn6.5}
\begin{vmatrix}
  f_0(\tau_0)&f_0(\tau_1)&\cdots&f_0(\tau_{d_w})\\
  f_1(\tau_0)&f_1(\tau_1)&\cdots&f_1(\tau_{d_w})\\
  \cdots&\cdots&\cdots&\cdots \\
  f_{d_w}(\tau_0)&f_{d_w}(\tau_1)&\cdots&f_{d_w}(\tau_{d_w})\\
\end{vmatrix}\ne 0.
\end{equation}

On the other hand, since
$\beta_{w}^-\alpha_{w+2}^-$ is an isomorphism,
and $\beta_{w}^-\alpha_{w+2}^-(f_j)=r^{-}(f_j)$, we can deduce that
  $$\left\{r^{-}(f_j)(p,q)\right\}_{j=0}^{d_w}$$
is a basis of $\mathcal{W}_w^-$.
Therefore, noting that $r_{2n}(f_j)\ne 0$, we know that
  $$\left\{\frac{r_{2n}(f_j)r^{-}(f_j)(p,q)}{(f_j,f_j)}\right\}_{j=0}^{d_w}$$
is also a basis of $\mathcal{W}_w^-$ and,
by \eqref{eqn6.5}, we conclude that
  $$\left\{\sum_{j=0}^{d_w}\frac{r_{2n}(f_j)r^{-}(f_j)(p,q)}{(f_j,f_j)}
  f_j(\tau_i)\right\}_{i=0}^{d_w}$$
is also a basis of $\mathcal{W}_w^-$.
This together with the identity \eqref{eqn6.4} imply that
 $$\{R_{2n}^-(p,q;\tau_i)\}_{i=0}^{d_w}$$
is again a basis of $\mathcal{W}_w^-$.

Finally, from the fact that $\beta_{w}^-$ is an isomorphism,
and that
$\beta_{w}^-(D_{2n}^-(p,q;\tau_i))=R_{2n}^-(p,q;\tau_i)$, we deduce that
 $$\{D_{2n}^-(p,q;\tau_i)\}_{i=0}^{d_w}$$
is a basis of $\mathcal{D}_w^-$. This establishes what we are after.
\end{proof}

\begin{rem}\label{rem6.1}
{\rm
It should be remarked that $R_{2n}^-(p,q;\tau)$ ``generates''
not only odd period polynomials but also modular forms. In other words,
it follows that
there are $(p_i,q_i)\in\CC^2\ (i=0,1,\ldots,d_w)$
such that
$\{R_{2n}^-(p_i,q_i;\tau)\}_{i=0}^{d_w}$ is a base of $M_{2n+2}$.
}
\end{rem}

\section{An application to Eisenstein series identities}\label{sect7}

In this section we will give an application of Theorem
\ref{thm1.1}.
Let $S_{2n}^-(p,q;\tau)$ and $T_{2n}^-(p,q;\tau)$ be defined by
\begin{equation*}
\begin{split}
  S_{2n}^-(p,q;\tau):=&R_{2n}^-(p,q;\tau)
    -\frac{1}{(2\pi i)^2}(2n+1)E_{2n+2}(\tau)\frac{1}{pq} \\
  =&-\frac{1}{(2\pi i)^2}\Bigg[
        \sum_{j=1}^{n}E_{2j}(\tau)E_{2n+2-2j}(\tau)
        p^{2j-1}q^{2n+1-2j} \\
        &\cccsp\ \ \ \ \ -E_{2n+2}(\tau)\frac{p^{2n+2}+q^{2n+2}}{pq}
     \Bigg] \\
  &-\frac{q}{2\pi ip}
    \frac{1}{2n}\frac{\partial E_{2n}(\tau)}{\partial \tau}p^{2n}
    -\frac{p}{2\pi iq}
    \frac{1}{2n}\frac{\partial E_{2n}(\tau)}{\partial \tau}q^{2n},
    \\
  T_{2n}^-(p,q;\tau):=&(2\pi i)^2pqS_{2n}^-(p,q;\tau). \\
\end{split}
\end{equation*}
By Theorem \ref{thm1.1}, $R_{2n}^-(p,q;\tau)$
satisfies the equation \eqref{eqn1.3}.
Since $1/pq$ also satisfies the equation \eqref{eqn1.3},
this can be carried over to $S_{2n}^-(p,q;\tau)$:
\begin{equation*}
S_{2n}^-(p+q,q;\tau)+S_{2n}^-(p,p+q;\tau)=S_{2n}^-(p,q;\tau).
\end{equation*}
Hence it follows that $T_{2n}^-(p,q;\tau)$ satisfies the equation
\begin{equation}\label{eqn7.1}
pT_{2n}^-(p+q,q;\tau)+qT_{2n}^-(p,p+q;\tau)
  =(p+q)T_{2n}^-(p,q;\tau).
\end{equation}

Now we set
$$c_j:=\begin{cases} E_{2n+2}(\tau),\ &j=0,n+1 \\
                    -E_2(\tau)E_{2n}(\tau)-\frac{\pi i}{n}
                    \frac{\partial E_{2n}(\tau)}{\partial \tau},
                    &j=1,n \\
                    -E_{2j}(\tau)E_{2n+2-2j}(\tau), &\text{otherwise}
      \end{cases}$$
so that $T_{2n}^-(p,q;\tau)$ can be expressed as
\begin{equation}\label{eqn7.2}
  T_{2n}^-(p,q;\tau)=\sum_{j=0}^{n+1}c_jp^{2j}q^{2n+2-2j}.
\end{equation}
This gives rise to the following Eisenstein series identities:
\begin{thm}\label{thm7.1}
For positive integers $n$ and $k$ with $1\leq k\leq 2n+2$, it holds that
\begin{equation}\label{eqn7.3}
  \sum_{\substack{i=0 \\ 2i\geq k-1}}^{n+1}
    \binom{2i}{k-1}c_i
  +\sum_{\substack{i=0 \\ 2i\leq k}}^{n+1}
    \binom{2n+2-2i}{2n+2-k}c_i
    =\begin{cases} c_{\frac{k-1}{2}},\ &\text{$k$\ odd} \\
                    c_{\frac{k}{2}},\ &\text{$k$\ even}.
      \end{cases}
\end{equation}
\end{thm}

\begin{proof}
>From \eqref{eqn7.1} and  \eqref{eqn7.2}, we have
\begin{align*}
  &p\sum_{i=0}^{n+1}c_i
    \sum_{j=0}^{2i}\binom{2i}{j}p^jq^{2i-j}q^{2n+2-2i}
  +q\sum_{i=0}^{n+1}c_ip^{2i}
    \sum_{j=0}^{2n+2-2i}\binom{2n+2-2i}{j}p^{2n+2-2i-j}q^j \\
  &\ccsp=p\sum_{i=0}^{n+1}c_ip^{2i}q^{2n+2-2i}
   +q\sum_{i=0}^{n+1}c_ip^{2i}q^{2n+2-2i}.
\end{align*}
By comparing the coefficients of $p^kq^{2n+3-k}$ in the both sides
of the equation above, we obtain the identities \eqref{eqn7.3}.
\end{proof}

If we take $k=1$ in Theorem \ref{thm7.1} we rediscover the formulas
  $$\frac{2\pi i}{n}\frac{\partial E_{2n}(\tau)}{\partial \tau}
  =-\sum_{j=1}^{n}E_{2j}(\tau)E_{2n+2-2j}(\tau)+(2n+3)E_{2n+2}(\tau)
  \ \ (n\geq 1)
  $$
which were proved by van der Pol \cite[p.\ 266]{PO1} and
Rankin \cite[Theorem 3]{RA2} (originated with Ramanujan \cite[p.\ 142]{RA1}).
Furthermore, Skoruppa \cite{SK1} discussed a method to produce such
identities for given $n$ and showed the first few of them.
On the other hand, our result \eqref{eqn7.3}
gives explicit formulas for any $n$.

{\bf Note added.} We are informed by Machide that his new result
\cite[Lemma 6.3]{MA2} implies that
$C(\tau)=-E_2(\tau)/(2\pi i)^2pq$ for the constant $C(\tau)$
in Theorem \ref{thm1.3}.

{\bf Acknowledgements.} The author would like to thank the referee
for helpful comments, and especially for providing him with
an elegant proof of Proposition \ref{prop3.1}.
He would also like to thank Professor N.~Yui and Dr. T.~Machide
for valuable comments.

\end{document}